\documentclass[12pt,a4paper]{article} 
\usepackage{amsmath}
\usepackage{amsfonts}
\usepackage{amssymb}
\usepackage{cite}
\usepackage{graphics}
\usepackage{amssymb,amscd}
\usepackage[utf8]{inputenc}
\usepackage[T2A]{fontenc}
\usepackage{tikz-cd}
\usepackage[russian,english]{babel}

\usepackage{amscd}
\usepackage{yfonts}
\usepackage[normalem]{ulem}
\usepackage{esint}
\usepackage{amsmath}
\usepackage{amsfonts}
\usepackage{amssymb}
\usepackage{hyperref}
\usepackage{verbatim}

\usepackage{graphics}
\usepackage{amsfonts,amssymb,amsmath,mathrsfs,amsthm,amscd}
\usepackage{tikz} % for pictures
\usepackage{tikz-cd} % for commutative diagrams
\usetikzlibrary{decorations.pathreplacing}

\numberwithin{equation}{subsection}

 \newtheorem{theorem}{Theorem}[section]
 \newtheorem{corollary}[theorem]{Corollary}
 
 \newtheorem{proposition}[theorem]{Proposition}

\theoremstyle{definition}

\newtheorem{remark}{Remark}
\newtheorem{example}{Example}

\newtheorem{definition}[theorem]{Definition}

\providecommand{\keywords}[1]{\textbf{\textit{Keywords:}} #1}

\begin{document}

\title{A note on deformations and mutations \\ of  fake weighted projective planes}
\author{\.{I}rem Portakal}

\newcommand{\Addresses}{{% additional braces for segregating \footnotesize
  \bigskip
  \footnotesize

  \textsc{Department of Mathematics, Otto-von-Guericke-University, Magdeburg, Germany}\par\nopagebreak
  \textit{E-mail address}, \.{I}.~Portakal: \texttt{irem.portakal@ovgu.de}

}}

\maketitle

\begin{abstract}
It has been shown by Hacking and Prokhorov that if the projective surface $X$ with quotient singularities and self-intersection number 9 has a smoothing to the projective plane, then $X$ is the general fiber of a $\mathbb{Q}$-Gorenstein deformation of the weighted projective plane with weights giving solutions to the Markov equation. This result has been understood and generalized by combinatorial mutations of Fano triangles by Akhtar, Coates, Galkin, and Kasprzyk. In this note, we study this result by utilizing polarized T-varieties and describe the associated deformation explicitly in terms of certain Minkowski summands of so-called divisorial polytopes. 
\end{abstract}


%\bodymatter

\section{Introduction}\label{Introduction}
In their paper \cite{hackingprokhorov}, Hacking and Prokhorov generalize Manetti's classification of surfaces $X$ with self intersection number $K_{X}\geq 5$. They classify all surfaces which admit a $\mathbb{Q}$-Gorenstein smoothing. As a corollary to their theorem, they show that if $X$ is a projective surface with quotient singularities and with $K_{X} ^ 2 =9$ which admits a smoothing to the projective plane, then $X$ is a $\mathbb{Q}$-Gorenstein deformation of the weighted projective plane $\mathbb{P}(a^2,b^2,c^2)$ where $(a,b,c)$ is a solution for the Markov equation, i.e. $a^2+b^2+c^2=3abc$ (Theorem \ref{motivation}). The Markov equation also appears in the context of derived categories. The solutions of the equation are in one-to-one correspondence with the exceptional bundles on the projective plane with slopes in the interval $[0,1/2]$ (\cite{rudakov}). \\

\noindent It turns out that we can understand this smoothing in terms of deformations of T-varieties. A normal variety $X$ is called a complexity-$d$ $T$-variety if it admits an effective algebraic torus action $T\cong (\mathbb{C}^*)^n$ such that $\text{dim}(X)-\text{dim}(T) = d$. The deformations of $T$-varieties have been studied by Altmann in \cite{alt1} and by Ilten and Vollmert in \cite{ilthrob1}. The key idea of the theory is to associate certain Minkowski decompositions to deformations. In this note, we consider the deformations of projective toric surfaces by interpretating them as complexity-one polarized $T$-varieties. To a polarized $T$-variety, one can associate a so-called divisorial polytope consisting of a lattice polytope and piecewise affine concave function. In \cite{ilthen1}, a technique has been introduced to deform these varieties by using admissible Minkowski summands of the associated piecewise affine concave function. The latest improvement on $\mathbb{Q}$-Gorenstein smoothings has been done combinatorially by Akhtar, Coates, Galkin and, Kasprzyk in \cite{mutations2} in terms of combinatorial mutations of Fano polytopes. This construction generalizes the result of Hacking and Prokhorov to fake projective planes satisfying a certain type of Diophantine equation. In fact, for this combinatorial mutation, one can relate a birational transformation. Let $f$ be a Laurent polynomial in $n$ variables. A mutation of $f$ is a birational transformation $\phi\colon(\mathbb{C}^{*})^n \dashrightarrow (\mathbb{C}^{*})^n$ such that it preserves the logarithmic volume form of the algebraic torus $(\mathbb{C}^{*})^n$ and $\phi^{*}f$ is again a Laurent polynomial. In \cite{ilthen3}, Ilten shows that for given such a mutation, there exists a flat projective family over $\mathbb{P}^1$ such that zero fiber is $\text{TV}(\triangle(f))$ and the infinity fiber is $\text{TV}(\triangle(\phi^*f))$. The proof involves the mentioned language of affine $T$-varieties via Minkowski decompositions by taking the cone over these projective varieties.\\
 
\noindent It is natural to try to relate the combinatorial mutations to deformations of projective varieties using lattice polytopes and T-varieties. Indeed, we associate a one-parameter deformation to a mutation of Fano triangle by utilizing divisorial polytopes. It gives an explicit exposition of the Minkowski summands of the associated deformation. In \cite{batyrev}, Baytrev proposed a method to construct Landau-Ginzburg model mirror dual to a smooth Fano variety $X$ as a Laurent polynomial $f$ such that the Newton polytope $\text{Newt}(f)$ is the fan polytope of the small toric degeneration of $X$. The choice of the Laurent polynomial $f$ is not unique, however one can always transform $f$ to another Laurent polynomial via mutations which is also mirror dual to $X$. Although the notion of small degenerations is restrictive, this model has been extended to any $\mathbb{Q}$-Gorenstein degeneration of the complex projective plane in \cite{galkin}. In particular, this is the case of Prokhorov-Hacking's deformations and it is another reason that we focus our attention into two dimensional case.\\

\noindent Remark that by toric degenerations we mean that the variety degenerates to a toric variety as one can guess. We say that $X$ degenerates to a toric variety $X_{0}$ if there exists a flat morphism over a smooth curve germ such that the general fiber is isomorphic to $X$ and the zero fiber is isomorphic to the toric variety $X_{0}$. This may be also expressed as the deformation of $X_0$ over a smooth curve germ.

\begin{definition}
Let $X$ be a normal surface with quotient singularities and Picard rank 1. We say that $X$ is a $\mathbb{Q}$-Gorenstein deformation of $X_0$ if there exists a deformation $\mathcal{X}$ over a smooth curve germ such that $X_{t}$ is isomorphic to $X$ for all $t\neq0$ and $K_{\mathcal{X}}$ is $\mathbb{Q}$-Cartier. If $X$ is smooth, the deformation is called $\mathbb{Q}$-Gorenstein smoothing.
\end{definition}

\noindent Prokhorov and Hacking prove the following result.

\begin{theorem}[\cite{hackingprokhorov}, Corollary 1.2]
Let $X$ be a projective surface with quotient singularities and $K_{X} ^ 2 =9$. If $X$ admits a smoothing to the complex projective plane, then $X$ is a $\mathbb{Q}$-Gorenstein deformation of $\mathbb{P}(a^2,b^2,c^2)$ where $a^2+b^2+c^2=3abc$.
\label{motivation}
\end{theorem}

\noindent We denote the one-parameter subgroups of an algebraic torus $T\cong (\mathbb{C}^*)^n$ by $N$ and the characters of $T$ by $M$. There is a natural bilinear pairing between these two lattices which is the usual dot product $${ \langle \bullet , \bullet \rangle} \colon M \times N \rightarrow \mathbb{Z}.$$ 

\noindent We let $N_{\mathbb{Q}} := N \otimes_{\mathbb{Z}} \mathbb{Q}$ and $M_{\mathbb{Q}}:= M \otimes_{\mathbb{Z}} \mathbb{Q}$ be the corresponding vector spaces to the lattices $N$ and $M$.

\section{Mutations of Fano polytopes}\label{mutations}

The notion of mutations of lattice polytopes has been introduced by Akhtar, Galkin, Coates, and Kaspryzk in their paper \cite{mutations2}. This construction is motivated by certain birational transformations of Laurent polynomials which have been studied in Galkin and Usnich's paper \cite{galkin}. Let $f \in \mathbb{C}[x_1^{\pm},\ldots,x_n^{\pm}]$ be a Laurent polynomial. The period of $f$ is defined as\\
$$\pi_f (t) := \left(\frac{1}{2\pi i}\right)^n \int_{|x_1|=\dots|x_n|=1} \frac{1}{1-tf(x_1,\ldots x_n)} \frac{dx_1}{x_1}\dots \frac{dx_n}{x_n}$$
\vspace{0.1cm}

\noindent where $t \in \mathbb{C}$ and $|t|\ll \infty$. A mutation of $f$, as in Definition 7 of \cite{galkin}, is a birational transformation $\phi \in \text{Aut} (\mathbb{C}(x_1, \dots, x_n))$ preserving the period $\pi_f$ such that $\phi(f)$ is again a Laurent polynomial. We are in particular interested in the following mutation.
\begin{example}\label{birationalmutation}
Let $g \in \mathbb{C}[x_2^{\pm},\ldots,x_n^{\pm}]$ be a Laurent polynomial. The birational transformation
$\phi: (x_1, \ldots, x_n) \mapsto (x_1/g, x_2, \ldots, x_n)$ is a mutation of $f$ if and only if $f$ can be written as $\sum_{i=k}^l f_i x_{1}^i$ where $f_i \in  \mathbb{C}[x_2^{\pm},\ldots,x_n^{\pm}] $ such that for $i>0$, $f_i/g^i \in \mathbb{C}[x_2^{\pm},\ldots,x_n^{\pm}]$.
\end{example}

\noindent Mutations of Laurent polynomials induce transformations of their associated Newton polytopes, which we call \emph{combinatorial mutations}. Moreover, if $0 \in \text{Newt}(f)$, then above-presented type of mutation of $f$ has been associated to the deformation of the toric variety of $\text{Newt}(f)$. Note that in this case, we require that $k<0$ and $l>0$. We will explain this relation in Section \ref{comparsec}. 

\subsection{Combinatorial mutations}
\noindent Let $P \subset N_{\mathbb{Q}}$ be a full dimensional convex lattice polytope. We say $P$ is \emph{Fano}, if the origin lies in the strict interior of $P$ and the vertices of $P$ are primitive lattice points. The dual of $P$ is defined as 
$$P^{*} = \{u \in M_{\mathbb{Q}} \ | \ \langle u, v \rangle \geq -1, \ \forall v \in P \} \subset M_{\mathbb{Q}}.$$
\noindent For the construction of combinatorial mutations, we follow \cite{mutations,mutations2}.\\

\noindent Let $w \in M$ be a primitive lattice vector. We set\\
$$h_{\text{min}}:= \text{min}\{ \langle w, v\rangle | v\in P\}, \ \   h_{\text{max}}:= \text{max}\{ \langle w, v\rangle | v\in P\}.$$
\vspace{0.1cm}

\noindent Since $P$ is a Fano polytope, we obtain that $h_{\text{min}}<0$ and $h_{\text{max}}>0$. Note that the lattice height function $w: N \rightarrow \mathbb{Z}$ naturally extends to $w_{\mathbb{Q}}: N_{\mathbb{Q}} \rightarrow \mathbb{Q}$. 

\begin{definition}
We say that a lattice point $v \in P$ is at height $m$ with respect to $w \in M$, if $\langle w,v \rangle = m$. For each height $h \in \mathbb{Z}$, we define a hyperplane $$H_{w,h} := \{x \in N_{\mathbb{Q}} \ | \ \langle w, x \rangle =h\}.$$
\end{definition}
\noindent We let $w_h(P) = \text{conv}(H_{w,h} \cap P \cap N)$.
\begin{definition}
A factor of $P$ with respect to $w$ is a lattice polytope $F \subset N_{\mathbb{Q}}$ satisfying:

\begin{enumerate}
\item $\langle w, v \rangle = 0$ for all $v \in F$.
\item For every $h \in \mathbb{Z}$ such that $h_{\text{min}} \leq h < 0$, there exists a (possibly empty) lattice polytope $G_h \subset N_{\mathbb{Q}}$ at height $h$ such that 
$$H_{w,h} \cap \text{vert}(P) \subseteq G_h + (-h)F \subseteq w_h(P).$$
\end{enumerate}
\end{definition}

\begin{remark}
A factor does not need to exist for every $w \in M$. If it exists, then we are in a position to define a combinatorial mutation of $P$. For a given combinatorial mutation of $P$, there exists a Laurent polynomial $f$ with its Newton polytope $\text{Newt}(f)$ to be equal to $P$ and the combinatorial mutation arises from an algebraic mutation. 
\end{remark}

\begin{definition}
The combinatorial mutation of $P$ with respect to height function $w$, factor $F$, and lattice polytopes $\{G_h\}$ is defined as the convex lattice polytope\\
$$\text{mut}_w(P,F;\{G_h\}) := \text{conv} \left(\bigcup^{-1}_{h=h_{\text{min}}} G_h \cup \bigcup_{h=0}^{h_{\text{max}}}(w_h (P) + hF  )\right) \subset N_{\mathbb{Q}}.$$
\end{definition}
\vspace{0.1cm}
\noindent This definition might appear technical at first glance. Let us construct an explicit combinatorial mutation induced by the birational transformation in Example \ref{birationalmutation}.

\begin{example}
Let $f= \sum_{i=k}^l f_i x_2^{i} \in \mathbb{C}[x_1^{\pm},x_2^{\pm}]$ with $k<0$ and $l>0$, and let $f_i, g \in \mathbb{C}[x_1^{\pm}]$ be Laurent polynomials for $i \in [k,l]$. Suppose that $g^i  \ | \ f_i$ for $i>0$. Consider the mutation $$\phi \colon (x_1,x_2) \rightarrow (x_1, x_2/g).$$ The Laurent polynomial $\phi^{*}f$ is the algebraic mutation of $f$ and can be written as $\sum_{i=k}^l (f_i/g^i) x_2^i \in \mathbb{C}[x_1^{\pm},x_2^{\pm}]$. Let $P:=\text{Newt}(f) \subset N_{\mathbb{Q}}$ and $Q:=\text{Newt}(\phi^* f) \subset N_{\mathbb{Q}}$. We set the height function for the induced combinatorial mutation $w=[0,1] \in M$ and the factor with respect to $w$ as $F = \text{conv}(0, (1,0)) \subset N_{\mathbb{Q}}$. First infer that, for $k \leq h \leq -1$, the lattice polytopes $G_h$ are $\text{Newt}(f_h g^hx_2^ h)$. Moreover, by the definition of the factor, we have the condition $\text{Newt}(f_h g^hx_2^ h) + [0,h] \subseteq w_h(P)$. Hence, we obtain that $g = 1 +x_1$ and observe that $F=\text{Newt}(g)$. 

\label{combicorres}
\end{example}

\noindent In fact, a combinatorial mutation $P \subset N_{\mathbb{Q}}$ induces a piecewise linear transformation $\varphi$ of $P^{*} \subset M_{\mathbb{Q}}$ such that $(\varphi(P^{*}))^{*} = \text{mut}_w(P,F;\{G_h\})$ and it is given by $$u \mapsto u-u_{\text{min}}w$$ where $u_{\text{min}} = \text{min}\{\langle u, v_F \rangle\ | \ v_F \in \text{vert}(F)\}$.

\subsection{Mutations of Fano triangles}
In this section, we bring our attention to the two dimensional case, more particularly to combinatorial mutations of Fano triangles. Let $P:= \text{conv}(v_1,v_2,v_3) \subset N_{\mathbb{Q}}$ be a Fano triangle. Since the origin is contained in the interior of $P$, there exists a unique choice of coprime positive integers $(\lambda_1, \lambda_2, \lambda_3)$ such that $\lambda_1 v_1 + \lambda_2 v_2 + \lambda_3 v_3 = 0$. Let $N' \subset N$ be the sublattice generated by the lattice points $v_i$. The associated projective toric surface $\text{TV} (P)$ is the {\it fake projective plane} with weights $\lambda_1, \lambda_2, \lambda_3$, i.e.\ it is the quotient of $\mathbb{P}(\lambda_1, \lambda_2, \lambda_3)$ by the action of the finite group $N/N'$ acting free in codimension one.   

\begin{figure}
\begin{center}
\includegraphics[width=4in]{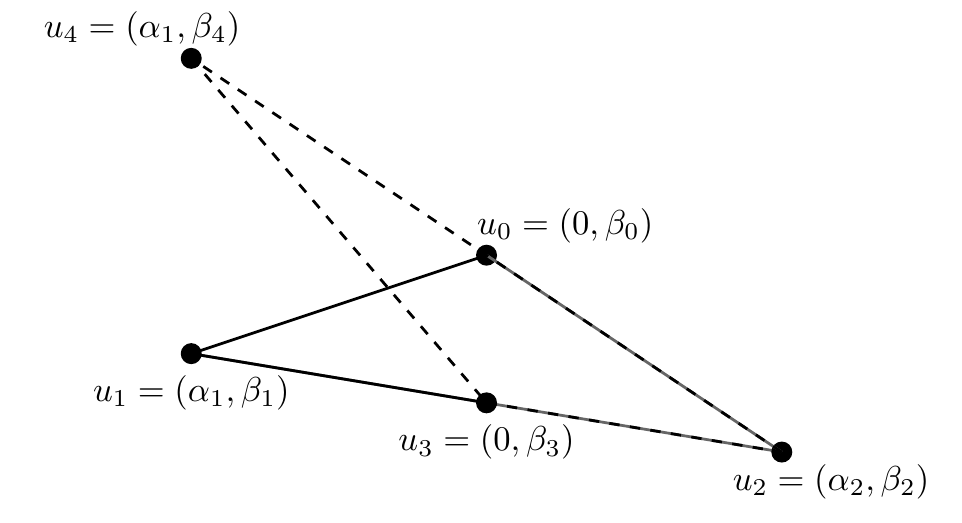}
\caption{Mutation $\Psi$ of a dual of a Fano triangle in $M_{\mathbb{Q}}$}
\label{fig1}
\end{center}
\end{figure}

\begin{proposition}[\cite{mutations}, Proposition 3.3]
Let $P \subset N_{\mathbb{Q}}$ be a Fano triangle and $\text{TV}(P)$ be the fake weighted projective plane with weights $(\lambda_1,\lambda_2,\lambda_3)$. Suppose that there exists a mutation of $P$ as $Q = \text{mut}_w(P,F; \{G_h\})$ for some height function $w$ and factor $F$ such that $\text{TV}(Q)$ is  a fake weighted projective plane. Then, up to relabelling, $\lambda_1 \ | \ (\lambda_2 + \lambda_3)^2$ and $\text{TV}(Q)$ has weights
$$\left(\lambda_2, \lambda_3, \frac{(\lambda_2 + \lambda_3)^2}{\lambda_1}\right).$$
\label{weighttheorem}
\end{proposition}

\noindent As a consequence to this proposition, one can show that the weights of $\text{TV}(P)$ and $\text{TV}(Q)$ belong to the solution set of the Diophantine equation \begin{equation} mx_1x_2x_3 = k(c_1 x_1^2 + c_2 x_2^2 + c_3x_3^2)\label{diophantine}\end{equation} where $m,k,c_i \in \mathbb{Z}_{>0}$ and $c_i$ is square free. Remark that a special case of this equation is the Markov equation which has been mentioned in the introduction.

\begin{proposition}[\cite{mutations}, Proposition 3.12]\label{mutationtriangle}
Let $P$ and $Q$ be Fano triangles as in Proposition \ref{weighttheorem}. Then the weights of $\text{TV}(P)$ and $\text{TV}(Q)$ give solutions to the same Diophantine equation (\ref{diophantine}). 
\end{proposition}

\noindent In the next sections, we study how these mutations give rise to deformations and hence this previous result can be  considered as a generalization of Theorem 1.1 by Hacking and Prokhorov in \cite{hackingprokhorov}. 

\begin{example}
Consider the combinatorial mutation $\Psi$ of the triangle $P^{*} \subset M_{\mathbb{Q}}$ to the dashed triangle $Q^{*} \subset M_{\mathbb{Q}}$ in Figure \ref{fig1}. It is induced by the mutation of the type of Example \ref{birationalmutation} with $w= (0,1) \in M$ and $F = \text{conv}(0,(1,0))$. In particular, if $\text{TV}(P)$ has the weights $(a^2,b^2,c^2)$, then $\text{TV}(Q)$ has weights $(b^2,c^2, \frac{(b^2+c^2)^2}{a^2})$. Suppose now $(a,b,c)$ is a solution to the Markov equation. Since the triple $(b,c, 3bc-a)$ is also a solution to Markov equation, we obtain\\
$$(b^2,c^2, \frac{(b^2+c^2)^2}{a^2})= (b^2,c^2, \frac{(3abc-a^2)^2}{a^2})=(b^2,c^2, (3bc-a)^2).$$
\vspace{0.1cm}

\noindent and thus the weights of $\text{TV}(Q)$ also give solution to the Markov equation.
\label{weightsmutation}
\end{example}

\section{Deformations of T-varieties}\label{deformations}
T-varieties naturally appear during the study of deformations of toric varieties preserving the torus action of the embedded torus. In this case, the total space of a deformation of a toric variety is not always toric, but they admit a lower dimensional effective torus action. \iffalse Let us recall the definition of a $T$-variety.
\begin{definition}
A normal variety is called a complexity-$d$ T-variety, if it admits an effective torus action $T \times X \rightarrow X$ such that $\text{dim}(X) - \text{dim}(T) =d$.
\end{definition}
\fi
The natural question to ask is if these varieties also admit a nice combinatorial description as in the toric case. A such combinatorial construction of T-varieties has been introduced in \cite{althaus}. The affine T-varieties are in one-to-one correspondence with the so-called p-divisors. A p-divisor is a formal product of a set of polyhedra and divisors from the ``good" quotient of the T-variety by the effective torus action. For a detailed read, one can refer to the survey on the language of $T$-varieties \cite{althauspeterilten}. In this section, we restrict our attention to the combinatorial construction for polarized complexity-one T-varieties. Next, we study one-parameter deformations of these varieties in terms of admissible Minkowski decompositions.

\subsection{Polarized T-varieties and Divisorial Polytopes}

The correspondence between polarized toric varieties and lattice polytopes in toric geometry has been generalized to the complexity-one T-varieties case in terms of so-called divisorial polytopes in \cite{ilthsuess}. In this section, we present this construction and in particular we apply this construction explicitly to fake weighted projective planes. 

\begin{definition}
A divisorial polytope $(\square, \Phi)$ consists of a lattice polytope $\square \subset M'_{\mathbb{Q}}$ and a piecewise linear concave function $$\Phi=\sum_{P \in \mathbb{P}^1} \Phi_P . P \colon \square \rightarrow \text{Div}_\mathbb{Q} \mathbb{P}^1
$$ such that 
\begin{enumerate}
\item $\text{deg } \Phi(u) > 0$ for $u$ in the interior of $\square$.
\item $\text{deg } \Phi(u) > 0 $ or $\Phi(u) \sim 0$ for $u$ a vertex of $\square$.
\item For all $P \in \mathbb{P}^1$, the graph of $\Phi_P$ has its vertices in $M' \times \mathbb{Z}$.
\end{enumerate}
\end{definition}

\begin{theorem}[\cite{ilthsuess}, Theorem 3.2]
There is a one-to-one correspondence between divisorial polytopes and complexity-one T-varieties $(X,\mathcal{L})$ where $\mathcal{L}$ is an equivariant ample line bundle.
\end{theorem}

\noindent For more details of this construction, we refer the reader to \cite{ilthsuess, ilthen1}. One can always consider a polarized toric variety with a one-dimension lower subtorus action. Let $\triangle \subset M_{\mathbb{Q}}$ be a lattice polytope in $M$. Then there exists the exact sequence of lattices

$$0 \rightarrow \mathbb{Z} \xrightarrow{\text{F}} M \xrightarrow{\text{deg}} M' \rightarrow 0.$$

\noindent We choose a section $s\colon M' \rightarrow M$, i.e.\ $\text{deg} (s) = \text{id}_{M'}$. Consider the map $\Phi_{\triangle} \colon \text{deg}(\triangle) \rightarrow \text{Div}_{\mathbb{Q}}(\mathbb{P}^1)$ given by 

$$(\Phi_{\triangle})_0(u) = \text{max} \{x \in \mathbb{Q} \ | \ F_{\mathbb{Q}} (x) + s(u) \in \triangle \cap \text{deg}_{\mathbb{Q}}^{-1}(u)\},$$
$$(\Phi_{\triangle})_{\infty}(u) = -\text{min} \{x \in \mathbb{Q} \ | \ F_{\mathbb{Q}} (x) + s(u) \in \triangle \cap \text{deg}_{\mathbb{Q}}^{-1}(u)\}.$$

\vspace{0.1cm}

\noindent Then, $(\Phi_{\triangle}, \text{deg}(\triangle))$ is a divisorial polytope with respect to the restricted effective torus action $T_{N'}$. 
\begin{figure}
\begin{center}
\includegraphics[width=4.5in]{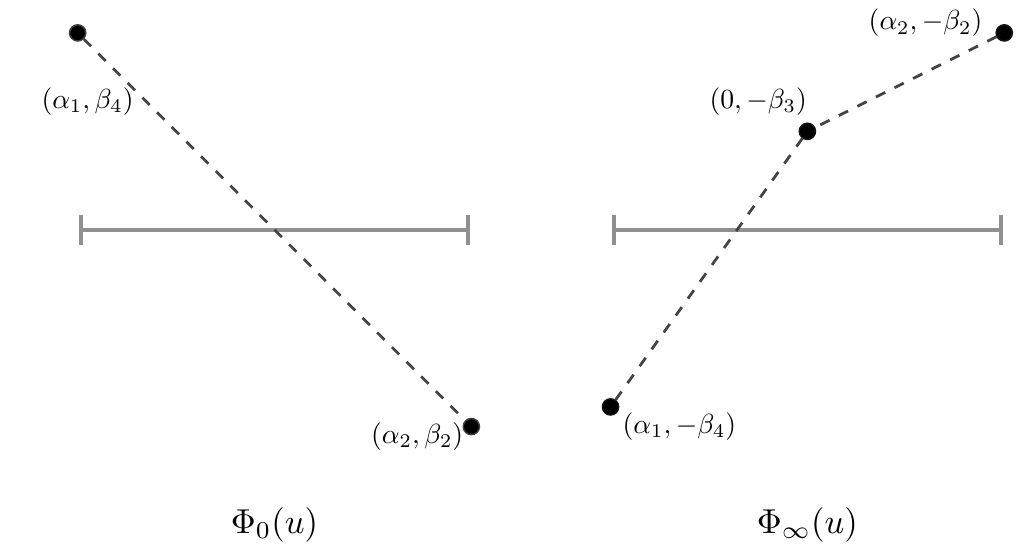}
\caption{The divisorial polytope associated to fake projective plane with weights $(\lambda_1, \lambda_2, \lambda_3)$}
\label{fig2}
\end{center}
\end{figure}

\begin{example}
Consider the dashed triangle $P^{*} \subset M_{\mathbb{Q}}$ from Figure \ref{fig1}. Set $\text{deg} \colon M \xrightarrow{(1,0)} M'$ to be the restricted torus action and choose the section $s \colon M' \xrightarrow{(1,0)^{\intercal}} M$. Note that $P^{*}$ does not need to be a lattice polytope. Let us suppose for the moment that $P^{*}$ is lattice polytope. Then we obtain the lattice polytope $\text{deg}(P^*)=\square \subset M'_{\mathbb{Q}}$ to be the interval $[\alpha_1, \alpha_2]$ and the piecewise linear concave functions $\Phi_0$ and $\Phi_{\infty}$ as in Figure \ref{fig2}.
\label{divisorialpolytopeexample}
\end{example}

\subsection{One-parameter deformations of polarized $T$-varieties}

\noindent A {\em deformation} of a projective algebraic variety $X_0$ is a flat map $\pi \colon \mathcal{X} \longrightarrow S$ with $0 \in S$ such that $\pi^{-1}(0) = X_0$. The variety $\mathcal{X}$ is called the total space and $S$ is called the base space of the deformation. If $S$ is a open subset of $\mathbb{A}^1$, we call $\pi$ a one-parameter deformation. Since we consider polarized $T$-varieties, we will study embedded deformations in this section. Let $X_0 \hookrightarrow \mathbb{P}^N$ be a projectively embedded variety. An {\em embedded deformation} $\pi$ of $X_0$ consists of $\mathcal{X} \hookrightarrow \mathbb{P}_S^N$ such that the projection to $S$ is a deformation of $X_0$ and the embedding of $\mathcal{X}$ restricts to the embedding of $X_0$. Our aim is to construct certain embedded one-parameter deformations of polarized $T$-varieties in terms of Minkowski summands. 
\begin{definition}
Let $(\square,\Phi)$ be a divisorial polytope. An admissible one-parameter Minkowski decomposition of $(\square,\Phi)$ consists of two piecewise affine functions $\Phi_P^0, \Phi_P^1 \colon \square \rightarrow \mathbb{Q}$ for some $P \in \text{Div}_{\mathbb{Q}} \mathbb{P}^1$ such that:
\begin{enumerate}
\item The graph of $\Phi_P^i$ has lattice vertices for $i=0,1$.
\item $\Phi_P(u) = \Phi_P^0 (u) + \Phi_P^1(u)$ for all $u \in \square$.
\item For any full-dimensional polytope $\triangle \subset \square$ on which $\Phi_P$ is affine, $\Phi_P^i$ has non-integral slope for at most one $i \in \{0,1\}$.
\end{enumerate}
\end{definition}
\vspace{0.2cm}

\noindent Given a one-parameter Minkowski decomposition of $(\square, \Phi)$, for any $s\in S$, define the divisorial polytope $\Phi^{\text{(s)}} \colon \square \rightarrow \text{Div}_{\mathbb{Q}} \mathbb{P}^1$ by

$$\Phi^{(s)} (u)= \sum_{P' \neq P} \Phi_{P'} (u) \otimes \mathbb{V}(y_{P'}) + \Phi_P^0(u) \otimes \mathbb{V}(y_P) +  \Phi_P^{1}(u) \otimes \mathbb{V}(y_P - s) $$

\noindent where $y_P \in \mathbb{C} (\mathbb{P}^1)$ is a rational function with its sole zero at $P$. Denote $X(\Phi)$ and  $X(\Phi^{(s)})$ to be the associated polarized $T$-varieties to divisorial polytopes $(\square, \Phi)$ and $(\square, \Phi^{(s)})$.

\begin{theorem}[\cite{ilthen1}, Theorem 7.3.2]
There exists a flat family of T-varieties $\{X(\Phi^{(s)})\}$ over $S$. Moreover, if $\Phi$ is very ample and gives a projectively normal embedding, this deformation can be realized as an embedded deformation of $X(\Phi^{(0)}) \cong X(\Phi)$.
\end{theorem}

\noindent The deformations of above form are called $T${\em-deformations}. These deformations admit the torus $T_N$ of $X(\Phi^{(0)})$ acting on the total space $\mathcal{X}$ and (trivially) on the base space $S$. Moreover, the maps $\pi \colon \mathcal{X} \rightarrow S$ and $X_0 \hookrightarrow X$ are $T_N$-equivariant. 

\section{Comparaison of mutations and deformations}\label{comparsec}

We now recall the mutation of the type of Example \ref{birationalmutation}. Suppose that $\text{Newt}(f)$ contains the origin in its interior. 
\begin{theorem}[\cite{ilthen3}, Theorem 1.3] 
There is a flat projective family $\pi \colon \mathcal{X} \rightarrow \mathbb{P}^1$ such that $\pi^{-1}(0)=\text{TV}(\text{Newt}(f))$ and $\pi^{-1}(\infty)=\text{TV}(\text{Newt}(\phi^* f))$.
\label{ilthencompar}
\end{theorem}

\noindent This deformation is constructed by taking the affine cone over the projective toric varieties and applying the techniques for the deformations of affine complexity-one $T$-varieties. Note also that, as we have shown in Example \ref{combicorres}, this mutation induces a combinatorial mutation. We utilize this combinatorial mutation to understand Proposition \ref{mutationtriangle} in terms of divisorial polytopes.

\begin{theorem}
For the combinatorial mutation $\Psi$, there exists an one-parameter embedded deformation of the polarized variety $\text{TV}(P)$ such that the general fiber is isomorphic to $\text{TV}(Q)$.  
\label{compar}
\end{theorem}

\begin{proof}
Let us consider the mutation $\Psi$ from the lattice polygon $P^* \subset M_{\mathbb{Q}}$ to the lattice polygon $Q^* \subset M_{\mathbb{Q}}$.  A lattice polygon in $M_{\mathbb{Q}}$ is very ample. Since $P^{*}$ is not necessarily a lattice polytope, we consider a dilation of $P^{*}$, i.e.\ the polarized toric variety $(\text{TV}(P), \mathcal{L}_{aP^*})$ for some $a \in \mathbb{Z}_{>0}$.  As we explained in Example \ref{weightsmutation}, the height function is $w=(0,1)$ and the factor is the line segment $F=\text{conv}(0,(x,0))$. Thus, we obtain the following piecewise linear transformation $v \mapsto v \Psi$ where $v=(v_1,v_2)$ and,\\
\begin{center}
 $\Psi=
\left\{
	\begin{array}{ll}
		 \begin{pmatrix}
  1& 0 \\
  0 & 1\\
 \end{pmatrix}  & \mbox{if } v_1>0 \\
		\begin{pmatrix}
  1& x \\
  0 & 1\\
 \end{pmatrix} &  \text{otherwise}
	.\end{array}
\right.$
\end{center}

\noindent Next, we utilize the associated divisorial polytope of $\text{TV}(P)$ from Example \ref{divisorialpolytopeexample}. The subtorus action is constructed by the map $M \xrightarrow{(1,0)} M'$. Without loss of generality, we assume that $(u_0,u_1,u_2)$ are lattice points of the lattice polytope $aP^{*} \in M_{\mathbb{Q}}$. In Figure \ref{fig2}, we observe that the divisorial polytope with respect to the restricted subtorus action consists of the line segment $[\alpha_1, \alpha_2]$ and two piecewise linear functions $\Phi_0$ and $\Phi_{\infty}$.\\
\begin{figure}
\begin{center}
\includegraphics[width=4.5in]{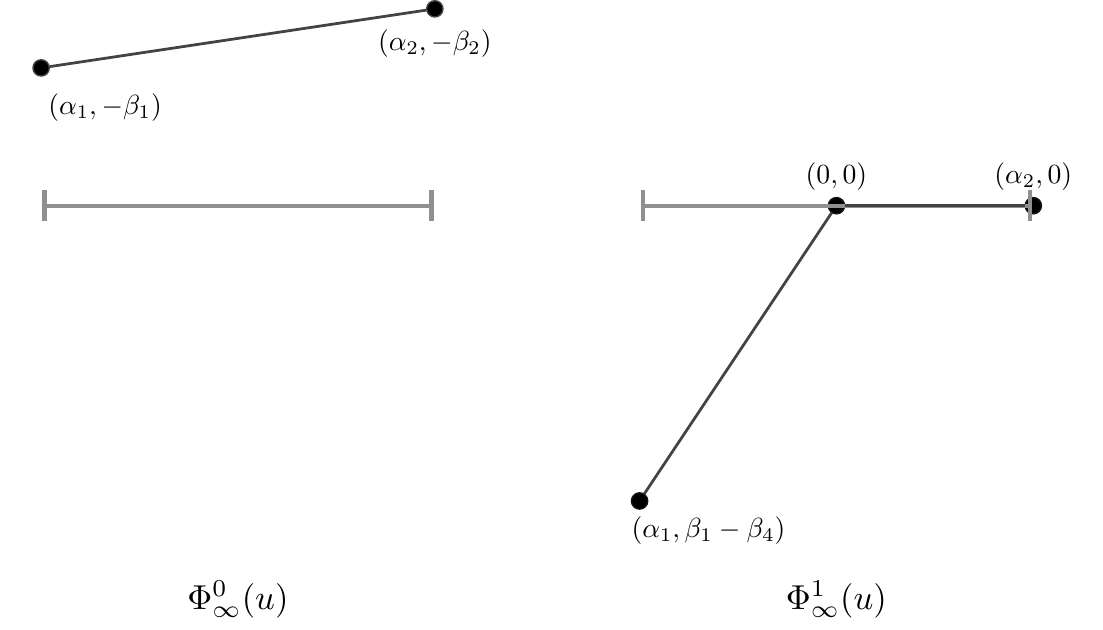}
\caption{An admissible one-parameter Minkowski decomposition of $\Phi_{\infty}$}
\label{fig3}
\end{center}
\end{figure}

\noindent In order to construct an embedded deformation of $\text{TV}(P)$, we examine the following admissible one-parameter Minkowski decomposition of $\Phi_\infty(u) = \Phi^{0}_{\infty}(u) + \Phi^{1}_{\infty}(u)$ as in Figure \ref{fig3}. This decomposition is indeed admissible, since we obtain an integral slope for $\Phi_{\infty}^1$ on interval $[\alpha_1,0]$: $$\frac{\beta_1 - \beta_4}{\alpha_1}=\frac{\beta_1 - \alpha_1 x -\beta_1}{\alpha_1}=-x.$$ Since we would like to obtain a toric variety as the general fiber of this deformation, the associated divisorial polytope $\Phi^{(s)}$ may have at most two nontrivial coefficient. Recall that the general fiber of this deformation is 
\vspace{0.2cm}
$$\Phi^{(s)} (u)=\Phi_{0} (u) \otimes \{0\} + \Phi_{\infty}^0(u) \otimes \mathbb{V}(y_{\infty}) +  \Phi_{\infty}^{1}(u) \otimes \mathbb{V}(y_{\infty} - s).$$
\vspace{0.05cm}

\noindent We get rid of the additional coefficient by adding the slope of $\frac{\beta_2 - \beta_4}{\alpha_2 - \alpha_1}$ of $\Phi_{0}$ to $\Phi_{\infty}^1$. Thus we obtain the coefficient for the divisor $\{0\} \in \text{Div}_{\mathbb{Q}}\mathbb{P}^1$ as the piecewise linear function with slopes $\frac{\beta_1 - \beta_0}{\alpha_1}$ on $[\alpha_1,0]$ and $\frac{\beta_2 - \beta_0}{\alpha_2}$ on $[0,\alpha_2]$. Hence, we conclude that $X(\Phi^{(s)})$ corresponds to the projective toric plane $\text{TV}(Q)$.

\end{proof} 
\vspace{0.2cm}
\noindent One can interpret these embedded deformations as deformations of so-called divisorial fans in lattice $N$. This construction can be found in \cite{ilthsuess}. The $T$-deformations of divisorial polytopes can be also seen as the $T$-deformations of affine cone over $\text{TV}(P)$ which uses a generalized method of deformations of toric varieties in \cite{alt1}. In fact, this is the method which has been used to prove Theorem \ref{ilthencompar}. Moreover, if the coefficients $\Phi_0$ and $\Phi_{\infty}^{i}$ form an admissible one-parameter Minkowski decomposition for some $i \in \{0,1\}$, then we can extend the deformation constructed in Theorem \ref{compar} to a deformation over $\mathbb{P}^1$. \\

\noindent From the proof of Theorem \ref{compar}, one deduces the following interesting fact.

\begin{corollary}
Let $\Phi_{\infty}=\Phi^0_{\infty} + \Phi^1_{\infty}$ be the admissible one-parameter Minkowski decomposition associated to the combinatorial mutation $\Psi$ as in Theorem \ref{compar}. Then $\Phi^0_{\infty}$ is a linear function and $\Phi^1_{\infty}$ has exactly two affine pieces with integral slopes. In particular, the decomposition of the slopes appear as\\
$$\frac{\beta_2-\beta_1}{\alpha_1 - \alpha_2} +\left\{0, \frac{\beta_1 - \beta_4}{\alpha_1}\right\}.$$
\end{corollary}

\begin{example}
Let $P^{*}$ be the dual polytope $\text{conv} ((-3,1),(3,1),(0,-1/2))$. Then $\text{TV}(P)$ is the weighted projective plane with weights $(1,1,4)$. Consider the ample line bundle $\mathcal{O}_{\mathbb{P}(1,1,4)}(4)$. We choose the restricted torus action again as $M \xrightarrow{(1,0)} M'$. Then the associated divisorial polytope consists of lattice polytope $[-6,6] \subset M'_{\mathbb{Q}}$ and the piecewise linear function as in Figure \ref{fig4}. To obtain a deformation with its general fiber isomorphic to $\mathbb{P}^2$, we set the decomposition of $\Phi_{\infty}$ as $\Phi_{\infty}^0$ with slope $-\frac{1}{2}$ and $\Phi_{\infty}^1$ with slopes $\{0,1\}$. It gives us a $\mathbb{Q}$-Gorenstein smoothing of $\mathbb{P}(1,1,4)$ to $\mathbb{P}^2$. 
\begin{figure}[h!]
\begin{center}
\includegraphics[width=4in]{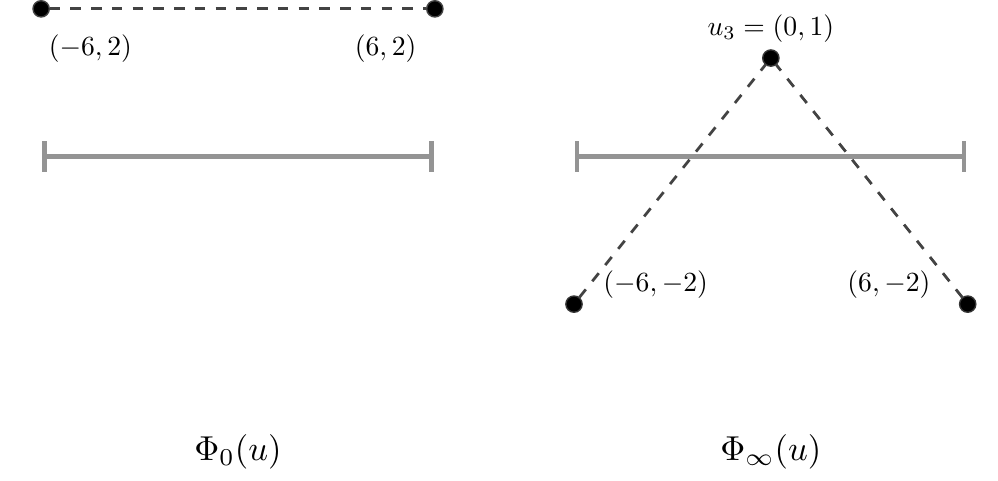}
\caption{A divisorial polytope for the weighted projective plane $\mathbb{P}(1,1,4)$}
\label{fig4}
\end{center}
\end{figure}
\noindent Now consider another admissible decomposition of $\Phi_{\infty}$ as in Figure \ref{fig5}. Note that this decomposition has slopes $\{1/2,0\}$ and $\{0,-1/2\}$. As in the proof of Theorem \ref{compar}, we construct the divisorial polytope $([-6,6], \Phi^{(s)})$. Thus, we obtain the general fiber $X(\Phi^{(s)})$ being isomorphic to $\mathbb{P}^1 \times \mathbb{P}^1$. In particular, this is not a $\mathbb{Q}$-Gorenstein deformation.

\begin{figure}
\begin{center}
\includegraphics[width=4in]{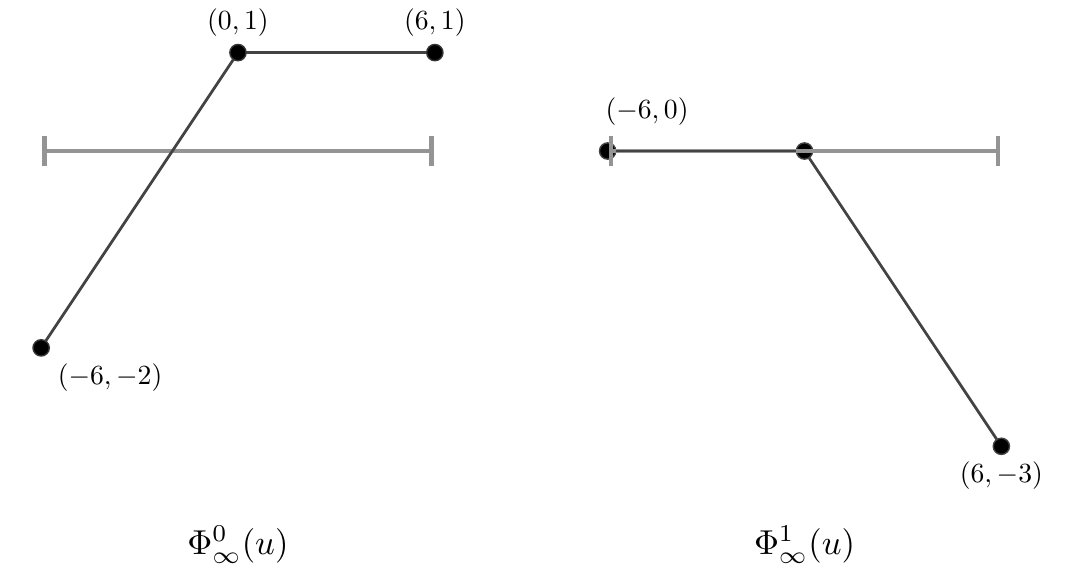}
\caption{A decomposition of $\Phi_{\infty}$ giving a deformation to $\mathbb{P}^1 \times \mathbb{P}^1$}
\label{fig5}
\end{center}
\end{figure}

\end{example}

%Non BiBTeX users can list down their references as:

\Addresses

\end{document}